\documentclass[reqno]{amsart}
\usepackage{cite}
\usepackage[russian,english]{babel}
\usepackage{fixltx2e}
\usepackage{amssymb,amsmath,amsfonts,amsthm}
\usepackage[utf8]{inputenc}
\usepackage{cmap}
\usepackage[T1]{fontenc}
\usepackage[all]{xy}
\usepackage{mathbbol}
\usepackage{graphicx}
\usepackage{xcolor}
\usepackage{lineno}

\newcommand{\black}{\color{black}}

\newcommand{\bl}{\color{blue}{[}\black }
\newcommand{\br}{\color{blue}{]}\black }
\newcommand{\rl}{\color{red}{[}\black }
\newcommand{\rr}{\color{red}{]}\black }

\DeclareMathOperator{\Span}{Span}

\newtheorem{theorem}[subsection]{Theorem}
\newtheorem{proposition}[subsection]{Proposition}

\newtheorem{corollary}[subsection]{Corollary}
\theoremstyle{definition}
\newtheorem{definition}[subsection]{Definition}
\newtheorem{example}[subsection]{Example}

\theoremstyle{remark}
\newtheorem{remark}[subsection]{Remark}

\begin{document}

\title[Some Leibniz bimodules of $\mathfrak{sl}_2$]{Some Leibniz bimodules of $\mathfrak{sl}_2$}
\author{T. Kurbanbaev}
\address{[Tuuelbay Kurbanbaev] Institute of Mathematics of Uzbek Academy of Sciences, Mirzo Ulugbek 81, 100170 Tashkent, Uzbekistan.}
\email{tuuelbay@mail.ru}
\author{R. Turdibaev}
\address{[Rustam Turdibaev] Inha University in Tashkent, Ziyolilar 9, 100170 Tashkent, Uzbekistan.}
\email{r.turdibaev@inha.uz}

\thanks{The authors were supported by the project \"E$\Phi$A-$\Phi$tex-2018-79.}

\begin{abstract}
    We study complex finite-dimensional Leibniz algebra bimodule over $\mathfrak{sl}_2$ that as a Lie algebra module is split into a direct sum of two simple $\mathfrak{sl}_2$-modules. We prove that in this case there are only two non-split Leibniz $\mathfrak{sl}_2$-bimodules and we describe the actions. 
\end{abstract}

\maketitle
\section*{Introduction}

Leibniz algebras were discovered by A. Bloh \cite{Bloh}. They are a non skew-symmetric generalization of Lie algebras and got popularity due to their rediscovery by J.–L. Loday \cite{LodayCyclic}. He noticed that a lifting of the classical Chevalley--Eilenberg boundary map to the tensor module of a Lie algebra gives rise to another chain complex. Originally born due to an observation in homology theory, this type of algebras became an object of an independent interest. 

For any Leibniz algebra $L$ there are notions (see \cite{Loday}, \cite{LP_Universal}) as an ideal generated by squares  of the elements $Leib(L)=\text{ideal}\langle [x,x] \mid x\in L\rangle$, also known as the Leibniz kernel and an associated Lie algebra $L/Leib(L)$ called the liezation which plays important role in the structural theory of Leibniz algebras. A notion of a simple Leibniz algebra does not exist in the classical algebraic meaning due to $Leib(L)$ being nontrivial ideal whenever $L$ is a non-Lie Leibniz algebra. However, a number of authors  use it for the Leibniz algebra with simple liezation  and simple ideal generated by the squares, introduced in \cite{Dzumadil'daev}. Therefore, there is a way to describe finite-dimensional simple Leibniz algebras over the complex numbers by a well-known classification of simple Lie algebras and their irreducible representations.

Representation or a bimodule of a Leibniz algebra $L$ is defined in \cite{LP_Universal} as a $\mathbb{K}$-module $M$ with two actions - left and right, satisfying compatibility conditions. It is established \cite{Barnes} that any simple finite-dimensional Leibniz representation is either symmetric, meaning the left and the right actions differ by sign, or antisymmetric, meaning  the left action is zero.  The classical Weyl's theorem on complete reducibility that claims any finite-dimensional module over a semisimple Lie algebra is a direct sum of simple modules does not generalize even for the simple Leibniz algebra case. In Proposition \ref{indecomposable_not_simple} we show that any non-Lie Leibniz algebra admits a Leibniz bimodule, which is neither simple, nor completely reducible. A similar example in dimension five is constructed in \cite{Fialowski}. This shows that the category of Leibniz representations of a Leibniz algebra is not semisimple.

A Lie algebra can be considered as a Leibniz algebra and one can consider Leibniz representation of a Lie algebra. In \cite{LP_Leib_rep} the authors describe the indecomposable objects of the category of Leibniz representations of a Lie algebra and as an example, in case the Lie algebra is $\mathfrak{sl}_2$ the indecomposable objects in that category can be described, whereas for $\mathfrak{sl}_n$ ($n>2$) they claim that it is of wild type. 

Our main focus in this work is to describe the actions of Leibniz representations of the Lie algebra $\mathfrak{sl}_2$ in details. Since a Leibniz bimodule over a Lie algebra satisfies the condition of a Lie module, in our study we use Weyl's complete reducibility to decompose Leibniz bimodule as a vector space into a direct sum of simple $\mathfrak{sl}_2$-modules. Our approach is different from homological approach of \cite{LP_Leib_rep} and uses direct calculations. Due to complexity of computations, we consider only the case when Leibniz bimodule as a vector space is a direct sum of two simple $\mathfrak{sl}_2$-modules and obtain the table of actions of  non-split indecomposable Leibniz $\mathfrak{sl}_2$-bimodules in Theorem \ref{main}. 

All modules and algebras in this work are finite-dimensional over a field $\mathbb{K}$ of characteristic zero. 
\section{Preliminaries}

\begin{definition}
    An algebra $(L,[\cdot,\cdot])$ over a field $\mathbb{K}$ is called a \emph{Leibniz algebra} if it is defined by the Leibniz identity
    $$[x,[y,z]]=[[x,y],z] - [[x,z],y], \ \mbox{for all}\ x,y \in L.$$
\end{definition}

\begin{definition}
    Let $M$ be a $\mathbb{K}$-vector space and bilinear maps $\bl-,-\br:L\times M \rightarrow M$ and $\rl-,-\rr:M\times L \rightarrow M$
    satisfy the following three axioms:
    \begin{align}
    \rl m,[x,y]\rr & =\rl\rl m,x\rr,y\rr-\rl \rl m,y\rr,x\rr,\label{red}\\
    \bl x,\rl m,y\rr\br & =\rl\bl x,m\br,y\rr-\bl [x,y],m\br,\label{bluered}\\
    \bl x,\bl y,m\br\br & =\bl[x,y],m\br-\rl\bl x,m\br,y\rr.\label{blueblue}
    \end{align}
    Then $M$ is called a \textit{representation} of the Leibniz algebra $L$ or an $L$\textit{-bimodule}.
\end{definition}
One can verify that on a direct sum $L\oplus M$ of vectors spaces the bracket
\begin{center}
	$\textbf{[}l+m, l'+m'\textbf{]}:=[l,l']+\rl m,l'\rr+\bl l,m'\br$
\end{center} 
defines a Leibniz algebra structure. Moreover, $L$ becomes a subalgebra, while $M$ becomes an ideal with $\textbf{[}M,M\textbf{]}=0$.
Note that, adding the identities $(\ref{bluered})$ and $(\ref{blueblue})$ one gets:
\begin{equation}
{\color{blue}[} x, {\color{red}[} m,y {\color{red}]} + {\color{blue}[} y,m
{\color{blue}]} {\color{blue}]} = 0.\label{blue}
\end{equation}
Consider a Lie algebra $\mathfrak{g}$ as a Leibniz algebra and let $M$ be its Leibniz representation. For the sake of convenience, we call a Leibniz algebra bimodule $M$ simply a bimodule $M$, and a Lie algebra module $N$, a $\mathfrak{g}$-module $N$. Note that, due to identity (\ref{red}) for a bimodule $M$, the bilinear map  $\rl -,-\rr : M\times \mathfrak{g} \to M$  is a right $\mathfrak{g}$-module action. In fact, a $\mathfrak{g}$-module with a left action satisfying (\ref{bluered}) and (\ref{blue}) is a Leibniz $\mathfrak{g}$-bimodule and the converse is also true. 

\begin{definition} Leibniz algebra $L$ is said to be \textit{simple} if the only ideals of $L$ are $\{0\},Leib(L),L$ and $[L,L]\neq Leib(L).$ 
\end{definition}
Obviously, in the case when Leibniz algebra is a Lie algebra, the ideal $Leib(L)$ is equal to zero. Therefore, this definition agrees with the definition of a simple Lie algebra.
Here is an example of simple Leibniz algebra from \cite{Dzumadil'daev}.

\begin{example}\label{simpleLeibniz} Let $\mathfrak{g}$ be a simple Lie algebra and $M$ be a simple right $\mathfrak{g}$-module. Then the vector space $L=\mathfrak{g}\oplus M$ equipped with the multiplication $\textbf{[}x+m,y+n\textbf{]}=[x,y]+\rl m,y\rr$, where $m,n \in M, \, x,y\in \mathfrak{g}$ is a simple Leibniz algebra.
\end{example}
\noindent Note that any finite-dimensional simple Leibniz algebra admits the construction of the Example \ref{simpleLeibniz}. 
\begin{definition}
	An $L$-bimodule is called \textit{simple} or \textit{irreducible}, if it does not admit non-trivial $L$-subbimodules. An $L$-bimodule is called \textit{indecomposable}, if it is not a direct sum of its $L$-subbimodules. An $L$-bimodule $M$ is called \textit{completely reducible} if for any $L$-subbimodule $N$ there exists a complementing $L$-subbimodule $N'$ such that $M=N\oplus N'$.
\end{definition}

\begin{proposition}\label{complete_reduciblity=complemented} An $L$-bimodule is a direct sum of simple $L$-submodules if and only if it is completely reducible.  
\end{proposition}
\begin{proof} Assume that $M=\oplus_{i=1}^n M_i$, where $M_i$ are simple $L$-subbimodules of $M$. Let $N$ be a non-trivial $L$-subbimodule of $M$ and consider $N\cap M_i$. Due to simplicity of $M_i$ either $N\cap M_i=\{0\}$ or $M_i\subseteq N$. Going through all $i\in \{1,\dots, n\}$ we obtain $\oplus_{i\in I}M_i\subseteq N$ for some non-empty subset $I\subseteq \{1,\dots, n\}$ and $M_j\cap N=\{0\}$ for all $j\in J=\{1,\dots n \}\setminus I$. Then $N'=\oplus_{j\in J}M_j$ is the  $L$-submodule of $M$ complementary to $N$. 
	
	Conversely, if $M$ is completely reducible, then either $M$ is simple or has a $L$-subbimodule, which is simple. Continuing the process for the complement of the simple $L$-subbimodule we obtain the direct sum decomposition of $M$. 
\end{proof}

Obviously, a simple $L$-bimodule is indecomposable. The converse holds for  semisimple  Lie algebras and Lie algebra modules.

\begin{theorem}(\cite{Jacobson})\label{completelyreducible} If $\mathfrak{g}$ is a finite-dimensional semi-simple Lie algebra over a field of characteristic 0, then every finite-dimensional module over $\mathfrak{g}$ is completely reducible.
\end{theorem}
Similar result does not hold for Leibniz algebras. An adjoint bimodule of a five-dimensional Leibniz algebra which is not a direct sum of simple bimodules is constructed in \cite{Fialowski}. Interestingly, the following generalization is true.  
\begin{proposition}\label{indecomposable_not_simple}
	Let $L$ be a non-Lie Leibniz algebra. The adjoint $L$-bimodule is neither completely reducible, nor simple. 
\end{proposition}
\begin{proof}
	Denote by $L^{ad}$ the adjoint $L$-bimodule with the underlying vector space $L$. Clearly, ideal of $L$ is equivalent to an $L$-subbimodule of $L^{ad}$ and therefore the Leibniz kernel $Leib(L)$ is an $L$-subbimodule of $L^{ad}$. Thus, $L^{ad}$ is not a simple $L$-bimodule. 
	
	Note that the Leibniz kernel $Leib(L)$ is not a direct summand of $L^{ad}$. Indeed, the liezation $L/Leib(L)$ is the compliment of $Leib(L)$ as a vector space, but cannot be an ideal of $L$. Otherwise $L$ is a Lie algebra, which is a contradiction.  By Proposition \ref{complete_reduciblity=complemented} it follows that $L^{ad}$ is not completely reducible. 
\end{proof}

Obviously, a simple $L$-bimodule is indecomposable, while the converse is not necessarily true by Proposition \ref{indecomposable_not_simple}. Indeed, for a simple Leibniz algebra $L$, the adjoint $L$-bimodule $L^{ad}$ admits only one nontrivial submodule - the Leibniz kernel $Leib(L)$. Hence, $L^{ad}$ is indecomposable, but not simple. 


\

Our goal in this work is to describe some indecomposable Leibniz bimodules of a Lie algebra $\mathfrak{sl}_2$ over a field of characteristic zero. Let $\{e,f,h\}$ be a basis of $\mathfrak{sl}_2$ in which it admits the following products:
\begin{center}
	$\begin{array}{lll}
[e,f]=h, & [e,h]=2e, & [f,h]=-2f.
\end{array}$
\end{center}

\begin{theorem}\label{irreducible}(\cite{Jacobson}) For each integer $m=0,1,2,\dots$ there exists one and, in the sense of isomorphism, only one irreducible $\mathfrak{sl}_2$-module $M$ of dimension $m+1.$ The module $M$ has a basis $\{x_0,x_1,\dots,x_m\}$ such that the representing transformations $E,F$ and $H$ corresponding to the canonical basis $\{e,f,h\}$ are given by:
	\begin{center}
		$\begin{array}{l}
	H(x_k)=(m-2k)x_k, \ k=0, \dots, m, \\
	F(x_m)=0, \ F(x_k)=x_{k+1}, \ k=0, \dots, m-1,  \\
	E(x_0)=0, \ E(x_k)=-k(m+1-k)x_{k-1}, \ k=1, \dots, m.\\
	\end{array}$
\end{center}
\end{theorem}

Consider a finite-dimensional Leibniz representation $M$ of $\mathfrak{sl}_2$. As a right $\mathfrak{sl}_2$-module, by Theorem \ref{completelyreducible} it is completely reducible into a direct sum of simple $\mathfrak{sl}_2$-modules $V_1\oplus\dots \oplus V_k$.  In the case $k=1$, obviously, $M$ becomes a simple bimodule. They are mentioned in \cite{LP_Leib_rep} as follows.

\begin{theorem} There exist two types of simple Leibniz bimodules over a Lie algebra $\mathfrak{g}$. The right action in both is a simple right $\mathfrak{g}$-module action, while the left action is either trivial or is the negative of the right action. 	
\end{theorem}

An $L$-bimodule with trivial left actions is called \textit{symmetric}. If the left action is the negative of the right action, then it is called \textit{antisymmetric}. Note that, a finite-dimensional simple $L$-bimodule is either symmetric or antisymmetric for any finite-dimensional Leibniz algebra $L$\cite[Lemma 1.9]{Barnes}.  

\

In this work we consider the case $k=2$ and our goal is to build extensions of simple object by another simple object in the category of finite-dimensional $\mathfrak{sl}_2$-bimodules.  Let $M$ be a direct sum of simple $\mathfrak{sl}_2$-modules $V$ and $W$. By Theorem \ref{irreducible}, there exist bases $\{v_0,v_1,\dots,v_{n}\}$ of $V$ and $\{w_0,w_1,\dots,w_m\}$ of $W$ such that
\begin{align*}
{\color{red}[}v_i , h{\color{red}]} &= (n-2i)v_i  & {\color{red}[}w_j, h{\color{red}]} &= (m-2j)w_j \\ 
{\color{red}[}v_{i} ,f{\color{red}]} &=  v_{i+1}  & {\color{red}[}w_{j},f{\color{red}]} &= w_{j+1}\\ 
{\color{red}[}v_{i} ,e{\color{red}]}& = -i(n-i+1)v_{i-1} & {\color{red}[}w_{j},e{\color{red}]} &= -j(m-j+1)w_{j-1}
\end{align*}
for all $0\leq i \leq n, \ 0\leq j \leq m$. Throughout the article we assume that $v_i=0$ and $w_j=0$ for all other values of $i$ and $j$. 

Since $V$ and $W$ do not have to be Leibniz $\mathfrak{sl}_2$-modules, we have
$$
\begin{array}{lll}
\bl h,v_i\br=\sum\limits_{j=0}^{n} \eta_{ij}^{11} v_j +\sum\limits_{j=0}^{m} \eta_{ij}^{12} w_j &\hspace*{1cm} \bl h,w_i\br =\sum\limits_{j=0}^{n} \eta_{ij}^{21} v_j +\sum\limits_{j=0}^{m} \eta_{ij}^{22} w_j\\ [2mm]
\bl f,v_i\br =\sum\limits_{j=0}^{n} \phi_{ij}^{11} v_j +\sum\limits_{j=0}^{m} \phi_{ij}^{12} w_j  &\hspace*{1cm} \bl f,w_i\br =\sum\limits_{j=0}^{n} \phi_{ij}^{21} v_j +\sum\limits_{j=0}^{m} \phi_{ij}^{22} w_j\\  [2mm]
\bl e,v_i\br =\sum\limits_{j=0}^{n} \epsilon_{ij}^{11} v_j +\sum\limits_{j=0}^{m} \epsilon_{ij}^{12} w_j &\hspace*{1cm} \bl e,w_i\br =\sum\limits_{j=0}^{n} \epsilon_{ij}^{21} v_j +\sum\limits_{j=0}^{m} \epsilon_{ij}^{22} w_j \\
\end{array}
$$
for some coefficients $\eta_{ij}^{11},\eta_{ij}^{12},\eta_{ij}^{21}, \eta_{ij}^{22}, 
\phi_{ij}^{11},\phi_{ij}^{12},\phi_{ij}^{21}, \phi_{ij}^{22}, 
\epsilon_{ij}^{11},\epsilon_{ij}^{12},\epsilon_{ij}^{21}, \epsilon_{ij}^{22}\in \mathbb{K}$. In the following sections we verify identities $(\ref{bluered})$ and $(\ref{blue})$ to find that many of those coefficients vanish. 

\section{Results from identity (\ref{bluered})}

Below we present the identities that are used in this section derived from the multiplication of $\mathfrak{sl}_2$ and identity (\ref{bluered}) for $v\in M$:
\begin{subequations}
	\renewcommand{\theequation}{5.\arabic{equation}}
\begin{align}
\label{[h,[m,h]]}
\bl h,\rl v,h \rr \br  & = \rl\bl h,v \br ,h\rr-\bl[h,h],v\br =\rl\bl h,v\br,h\rr
\\
 \label{[f,[m,h]]}
\bl f,\rl v,h\rr \br & = \rl\bl f,v\br ,h\rr-\bl[f,h],v\br = \rl\bl f,v\br ,h\rr+2\bl f,v\br
\\
 \label{[e,[m,h]]}
\bl e,\rl v,h\rr \br & = \rl\bl e,v\br ,h\rr-\bl[e,h],v\br =\rl\bl e,v\br ,h\rr-2\bl e,v\br
\\
\label{[f,[m,f]]}
\bl f,\rl v,f\rr \br & = \rl\bl f,v\br ,f\rr-\bl[f,f],v\br =\rl\bl f,v\br ,f\rr
\\
 \label{[e,[m,e]]}
\bl e,\rl v,e\rr \br & = \rl\bl e,v\br ,e\rr-\bl[e,e],v\br =\rl\bl e,v\br ,e\rr
\\
 \label{[f,[m,e]]}
\bl f,\rl v,e\rr \br & = \rl\bl f,v]\br ,e\rr-\bl[f,e],v\br  =\rl\bl f,v\br ,e\rr+\bl h,v\br
\\
\label{[h,[m,f]]}
\bl h,\rl v,f\rr \br & = \rl\bl h,v\br ,f\rr-\bl[h,f],v\br =\rl\bl h,v\br ,f\rr-2\bl f,v\br
\\
\label{[h,[m,e]]}
\bl h,\rl v,e\rr \br & = \rl\bl h,v\br ,e\rr-\bl[h,e],v\br =\rl\bl h,v\br ,e\rr+2\bl e,v\br
\end{align}
\end{subequations}
Without loss of generality we assume that $n\geq m$ and set $ \ell=\frac12(n-m)$.

\begin{proposition} \label{proposition1} The following holds:
 \begin{equation}\label{at_most_2_coefficients}
    \begin{array}{lllllll}
    \bl h, v_i \br = \eta_{i}^{11}v_i +\eta^{12}_{i-\ell}w_{i-\ell}, & \bl h, w_j\br = \eta_{j+\ell}^{21}v_{j+\ell} +\eta_{j}^{22}w_j,\\  
    \bl f, v_i \br = \phi_{i+1}^{11}v_{i+1} +\phi^{12}_{i+1-\ell}w_{i+1-\ell}, & \bl f, w_j\br = \phi_{j+1+\ell}^{21}v_{j+1+\ell} +\phi_{j+1}^{22}w_{j+1},\\ 
    \bl e, v_i \br = \epsilon_{i-1}^{11}v_{i-1} +\epsilon^{12}_{i-1-\ell}w_{i-1-\ell},  & \bl e,w_j\br =\epsilon_{j-1+\ell}^{21}v_{j-1+\ell} +\epsilon_{j-1}^{22}w_{j-1}, \\
    \end{array}
 \end{equation}
where $\eta^{12}_{i-\ell}, \eta_{j+\ell}^{21},  \phi^{12}_{i+1-\ell}, \phi_{j+1+\ell}^{21}, \epsilon^{12}_{i-1-\ell},\epsilon_{j-1+\ell}^{21}$ are zero if $\ell$ is not an integer. 
\end{proposition}

\begin{proof} Substituting $v=v_i$ in identities (\ref{[h,[m,h]]}), (\ref{[f,[m,h]]}) and (\ref{[e,[m,h]]}) yields the following:
 \begin{align*}
 (n-2i)(\sum\limits_{j=0}^{n} \eta_{ij}^{11} v_j +\sum\limits_{j=0}^{m} \eta_{ij}^{12} w_j) &= \sum\limits_{j=0}^{n} \eta_{ij}^{11} (n-2j) v_j +\sum\limits_{j=0}^{m} \eta_{ij}^{12}(m-2j) w_j;\\
  (n-2i)(\sum\limits_{j=0}^{n} \phi_{ij}^{11} v_j +\sum\limits_{j=0}^{m} \phi_{ij}^{12} w_j) &= \sum\limits_{j=0}^{n} \phi_{ij}^{11} (n-2j+2) v_j +\sum\limits_{j=0}^{m} \phi_{ij}^{12}(m-2j+2) w_j;\\
      (n-2i)(\sum\limits_{j=0}^{n} \epsilon_{ij}^{11} v_j +\sum\limits_{j=0}^{m} \epsilon_{ij}^{12} w_j)&=\sum\limits_{j=0}^{n} (n-2j-2)\epsilon_{ij}^{11} v_j +\sum\limits_{j=0}^{m}(m-2j-2) \epsilon_{ij}^{12} w_j.
 \end{align*}
 
We get $\eta_{ij}^{11}=0$ unless $j=i$ and $\eta_{ij}^{12}=0$ unless $j=i-\ell$; 

\hspace*{1.09cm} $\phi_{ij}^{11}=0$ unless $j=i+1$ and $\phi_{ij}^{12}=0$ unless $j=i+1-\ell$;

\hspace*{1.15cm} $\epsilon_{ij}^{11}=0$ unless $j=i-1$ and $\epsilon_{ij}^{12}=0$ unless $j=i-1-\ell$.

Denote by $\eta^{11}_i:=\eta_{i, i}^{11}$, $\eta^{12}_{i-\ell}:=\eta_{i, i-\ell}^{12}$, $\phi^{11}_{i+1}:=\phi_{i, i+1}^{11}$, $\phi^{12}_{i+1-\ell}:=\phi_{i, i+1-\ell}^{12}$, $\epsilon^{11}_{i-1}:=\epsilon_{i, i-1}^{11}$, $\epsilon^{12}_{i-1-\ell}:=\epsilon_{i, i-1-\ell}^{12}$ to obtain the first column of (\ref{at_most_2_coefficients}). By the symmetry, interchanging $v$'s and $w$'s one derives the same results, except $\ell$ is changed to $-\ell$, which is the second column of (\ref{at_most_2_coefficients}). 
\end{proof}

\begin{proposition}\label{proposition_f}
	For $0 \leq i \leq n-1 $ and $0\leq j \leq m-1$ we have the following 
	$$\begin{array}{lll}
\, \text{for } \ell=0\colon &            \vline  \hspace*{0.5cm}   \text{for }  \ell=1\colon&              \vline  \hspace*{0.5cm}    \text{for }   \ell\geq 2\colon        \\
\, \bl f,v_{i}\br =\phi^{11}v_{i+1}+\phi^{12}w_{i+1}   &   \vline   \hspace*{0.5cm}   \bl f,v_i\br =\phi^{11} v_{i+1}+\phi^{12} w_i &     \vline  \hspace*{0.5cm}     \bl f,v_i\br =\phi^{11} v_{i+1}                      \\
\, \bl f,w_{i}\br =\phi^{21}v_{i+1}+\phi^{22}w_{i+1}   &    \vline   \hspace*{0.5cm}  \bl f,w_j\br =\phi^{21}v_{j+2}+\phi^{22} w_{j+1}               &    \vline   \hspace*{0.5cm}     \bl f,w_j\br =\phi^{22} w_{j+1}\\
  &    \vline   \hspace*{0.5cm}  \bl f,w_m\br =\phi^{21}v_{n}              &    \vline   \hspace*{0.5cm}     
  \end{array}.	$$
\end{proposition}

\begin{proof} We consider several cases. 
	
	\textit{\textbf{Case 1.}} Let $\ell=0$. By Proposition \ref{proposition1} we have
	$$\begin{array}{lll}
	\bl f,v_i\br =\phi^{11}_{i+1}v_{i+1}+\phi^{12}_{i+1}w_{i+1}, & \bl f,w_i\br =\phi^{21}_{i+1}v_{i+1}+\phi^{22}_{i+1}w_{i+1}, & 0\leq i\leq n-1.
	\end{array}$$
	For all $0\leq i\leq n-1$ the identity (\ref{[f,[m,f]]}) with $v=v_i$ implies
	$$\phi^{11}_{i+2}v_{i+2}+\phi^{12}_{i+2}w_{i+2}=\phi^{11}_{i+1} v_{i+2}+\phi^{12}_{i+1}w_{i+2}.$$
	From this we have
	$\phi^{11}_{i+1}=\phi^{11}_{i+2}$ and $\phi^{12}_{i+1}=\phi^{12}_{i+2}$ for all $0\leq i\leq n-2.$ Set $\phi^{11}:=\phi^{11}_1$ and $\phi^{12}:=\phi^{12}_1$ to get $\bl f,v_i\br =\phi^{11} v_{i+1}+\phi^{12} w_{i+1}$ for all $0\leq i\leq n-1$. Analogously, $\bl f,w_i\br =\phi^{21} v_{i+1}+\phi^{22} w_{i+1}, \,\, 0\leq i\leq n-1.$
	
\textit{\textbf{Case 2.}} Let  $\ell=1$. By Proposition \ref{proposition1} there are the following brackets:
	\[
	\begin{array}{lll}
	\bl f,v_i\br =\phi^{11}_{i+1}v_{i+1}+\phi^{12}_i w_i, \ \ 0\leq i \leq m=n-2, \\ [1mm]
	\bl f,v_{n-1}\br =\phi^{11}_{n}v_{n},  \\ [1mm]
	\bl f,w_i\br =\phi^{21}_{i+2}v_{i+2}+\phi^{22}_{i+1}w_{i+1}, \ \ 0\leq i \leq m-1.
	\end{array}
	\]

As in Case 1 we obtain 
\begin{center}
		$\phi^{11}_1=\phi^{11}_2=\dots=\phi^{11}_{m+2}:=\phi^{11}, \quad \phi^{12}_0=\phi^{12}_1=\dots=\phi^{12}_m:=\phi^{12}$,
		
	$\phi^{21}_2=\phi^{21}_3=\dots=\phi^{21}_{m+1}:=\phi^{21}, \quad \phi^{22}_1=\phi^{22}_2=\dots=\phi^{22}_m:=\phi^{22}$.
\end{center}

\textit{\textbf{Case 3.}} Let $\ell \geq 2$. By Proposition \ref{proposition1} there are the following brackets:
	$$\begin{array}{lll}
	
	\bl f,v_{i-1} \br =\phi_{i}^{11}v_{i} , \,\, 0\leq i\leq \ell-1, & \\
	[2mm]
	
	\bl f,v_{i-1}\br =\phi_{i}^{11}v_{i} +\phi_{-\ell+i}^{12}w_{-\ell+i}, \,\, \ell\leq i\leq m+\ell, & \\ [2mm]
	
	\bl f,v_{i-1} \br =\phi_{i}^{11}v_{i} , \,\, m+\ell+1\leq i\leq n, & \\ [2mm]
	
	\bl f,w_{j-1}\br =\phi_{\ell+j}^{21}v_{\ell+j} +\eta_{j}^{22}w_{j}, \,\, 0\leq j\leq m. & \\ 
	\end{array}$$
	
As in Case 1 we obtain 	
\begin{center}
$\phi^{11}_1=\phi^{11}_1=\dots=\phi^{11}_{n}=\colon\phi^{11}, \quad \phi^{12}_0=\phi^{12}_1=\dots=\phi^{12}_m=\colon\phi^{12},$

$\phi^{21}_{\ell+1}=\phi^{21}_{\ell+2}=\dots=\phi^{21}_{m+\ell}=\colon\phi^{21}, \quad \phi^{22}_1=\phi^{22}_2=\dots=\phi^{22}_m=\colon\phi^{22}.$	
\end{center}	

Using the identity (\ref{[f,[m,f]]}) for $v=w_{\ell-2}$ and $v=w_{m-1}$ yields $ \phi^{12}=0$ and $ \phi^{21}=0$. This completes the proof.
\end{proof}

\begin{proposition}\label{proposition_e}
For $0 \leq i \leq n $ and $0\leq j \leq m$ we have the following
	$$\begin{array}{lll}
\text{for }	\ell=0\colon& \bl e,v_i\br =\displaystyle\frac{i(n-i+1)}{n}(\epsilon^{11}v_{i-1}+\epsilon^{12}w_{i-1}), \\ [1mm]
	& \bl e,w_i\br =\displaystyle\frac{i(n-i+1)}{n}(\epsilon^{21}v_{i-1}+\epsilon^{22}w_{i-1});  \\ [2mm]
	\text{for }	\ell= 1\colon & \bl e,v_i\br = i\left( \displaystyle\frac{(n-i+1)}{n}\epsilon^{11}v_{i-1}+ \frac{i+1}{2}\epsilon^{12}w_{i-2} \right), \\ [1mm]
	& \bl e,w_j\br =(m+1-j)\left(\displaystyle\frac{m+2-j}{(m+2)(m+1)}\epsilon^{21}v_j+\frac{j}{m}\epsilon^{22}w_{j-1}\right); \\ [1mm]
\text{for }	\ell\geq 2\colon & \bl e,v_i\br =\displaystyle\frac{i(n-i+1)}{n}\epsilon^{11}v_{i-1}, \\ [1mm]
	& \bl e,w_j\br =\displaystyle\frac{j(m-j+1)}{m}\epsilon^{22}w_{j-1}. \\ [1mm]
	\end{array}$$
\end{proposition}

\begin{proof} As in the proof of Proposition \ref{proposition_f}, we consider several cases. 
	
\textit{\textbf{Case 1.}} Let $\ell=0$. By Proposition \ref{proposition1} for $1\leq i\leq n$ we have 
	$$\begin{array}{ll}
		\bl e,v_i\br =\epsilon^{11}_{i-1} v_{i-1}+\epsilon^{12}_{i-1}w_{i-1}, & \bl e,w_i\br =\epsilon^{21}_{i-1}v_{i-1}+\epsilon^{22}_{i-1} w_{i-1}.
	\end{array}$$
	
	Consider the identity (\ref{[e,[m,e]]}) for $v=v_i$ and $1\leq i\leq n$ to get 	$$i(n-i+1)(\epsilon^{11}_{i-2}v_{i-2}+\epsilon^{12}_{i-2}w_{i-2})=(i-1)(n-i+2)(\epsilon^{11}_{i-1} v_{i-2}+\epsilon^{12}_{i-1}w_{i-2}).$$
	From this we obtain for $2\leq i\leq n$ the equalities
	$$\epsilon^{11}_{i-1}(i-1)(n-i+2)=\epsilon^{11}_{i-2}i(n-i+1) \text{ and } \epsilon^{12}_{i-1}(i-1)(n-i+2)=\epsilon^{12}_{i-2}i(n-i+1).$$
Denote by $\epsilon^{11}\colon=\epsilon^{11}_0$ and $\epsilon^{12}\colon=\epsilon^{12}_0$. Then for $1\leq i \leq n-1$ we have  
	$$\epsilon^{11}_i=\frac{(i+1)(n-i)}{n}\epsilon^{11}, \quad \epsilon^{12}_i=\frac{(i+1)(n-i)}{n}\epsilon^{12}.$$
	
\noindent Similarly, for $1\leq i \leq n-1$ one gets 
	$\displaystyle \epsilon^{21}_i=\frac{(i+1)(n-i)}{n}\epsilon^{21}$ and $\displaystyle \epsilon^{22}_i=\frac{(i+1)(n-i)}{n}\epsilon^{22}$. 
	
\textit{\textbf{Case 2.}} Let $\ell=1$. By Proposition \ref{proposition1} there are the following brackets:
	\[
	\begin{array}{lll}
	\bl e,v_1\br =\epsilon^{11}_0 v_0, \\ [1mm]
	\bl e,v_i\br =\epsilon^{11}_{i-1}v_{i-1}+\epsilon^{12}_{i-2}w_{i-2}, \ \ 2\leq i \leq m+2=n, \\ [1mm]
		\bl e,w_0\br=\epsilon^{21}_0v_0,\\[1mm]
	\bl e,w_i\br =\epsilon^{21}_iv_i+\epsilon^{22}_{i-1}w_{i-1}, \ \ 1\leq i \leq m.
	\end{array}
	\]
	As in the Case 1 we have 
	\[   \epsilon^{11}_i=\frac{(i+1)(n-i)}{n}\epsilon^{11}, \quad 1\leq i\leq n-1 ; \quad  \epsilon^{12}_i=\frac{(i+1)(i+2)}{2}\epsilon^{12}, \quad 1\leq i\leq m ; \]
	$$\epsilon_j^{21}= \frac{(m+1-j)(m+2-j)}{(m+2)(m+1)}\epsilon_0^{21},\  \epsilon^{22}_{j-1}= \frac{j(m+1-j)}{m}\epsilon^{22}, \ \text{for } 1\leq j \leq m.$$
	
\textit{\textbf{Case 3.}} Let $\ell \geq 2$. By Proposition \ref{proposition1} there are the following brackets:
	$$\begin{array}{lll}
	\bl e,v_{i+1} \br =\epsilon_{i}^{11}v_{i} , \,\, 0\leq i\leq \ell-1, \\ [1mm]
	\bl e,v_{i+1} \br =\epsilon_{i}^{11}v_{i} +\epsilon_{-\ell+i}^{12}w_{-\ell+i}, \,\, \ell\leq i\leq m+\ell, \\ [1mm]
	\bl e,v_{i+1} \br =\epsilon_{i}^{11}v_{i} , \,\, m+\ell+1\leq i\leq n-1, \\ [1mm]
	\bl e,w_{j+1}\br =\epsilon_{\ell+j}^{21}v_{\ell+j} +\epsilon_{j}^{22}w_{j}, \,\, 0\leq j\leq m. \\
	\end{array}$$

As in the previous cases using  the identity (\ref{[e,[m,e]]}) for $v=v_i, 1\leq i\leq n$ one obtains $\epsilon_i^{11}=\displaystyle \frac{(i+1)(n-i)}{n}\epsilon_0^{11}$ and $\epsilon_{i+1-\ell}^{12}=\displaystyle \frac{\binom{m}{i-\ell-1}}{\binom{i+2}{\ell+1}\binom{m+\ell-3}{i-\ell-1}}\cdot \frac{1}{(m+\ell-2)(m+\ell-1)}   \epsilon_{0}^{12}$.

The identity (\ref{[e,[m,e]]}) for $v=v_{\ell+m+2}$ yields $\epsilon^{12}_m=0$, which implies $\epsilon^{12}_i=0$ for all $0\leq i \leq m$.

Similarly, the identity (\ref{[e,[m,e]]}) for $v=w_1$ derives $\epsilon^{21}_\ell=0$.  As in the previous cases, 
	$\epsilon^{21}_\ell=\epsilon^{21}_{\ell+1}=\dots=\epsilon^{21}_{\ell+m-1}=0$ and $\displaystyle \epsilon^{22}_i=\frac{(i-1)(m-i)}{m}\epsilon^{22}_0$ for all $1\leq i\leq m$. Denoting by $\epsilon^{22}:=\epsilon^{22}_0$ completes the proof. 
	
Note that, in cases when $m=0$ all multiplications vanish since there is no basis vector with index negative index. 	
\end{proof}

\begin{proposition} \label{proposition_h} We have 
	$$\begin{array}{llll}
	\text{for } \ell=0: & \bl h,v_i\br =(\eta^{11}-2i\phi^{11})v_i+(\eta^{12}-2i\phi^{12})w_i, \,\, 0\leq i\leq n, \\ [1mm]
	& \bl h,w_i\br =(\eta^{21}-2i\phi^{21})v_i+(\eta^{22}-2i\phi^{22})w_i, \,\, 0\leq i\leq n; \\ [1mm]
\text{for }	\ell=1: & \bl h,v_i\br =(\eta^{11}-2i\phi^{11})v_i-2i\phi^{12} w_{i-1}, \ \ 0\leq i \leq n, \\ [1mm]
	& \bl h,w_i\br =2(m+1-i)\phi^{21} v_{i+1}+(\eta^{22}-2i\phi^{22})w_i, \quad 0\leq i \leq m; \\ [1mm]
\text{for }	\ell\geq 2: & \bl h,v_i\br =(\eta^{11}-2i\phi^{11})v_i, \ \ 0 \leq i \leq n, \\ [1mm]
	& \bl h,w_i\br =(\eta^{22}-2i\phi^{21})w_i, \ \ 0 \leq i \leq m.
	\end{array}
	$$
\end{proposition}

\begin{proof} Consider the following cases.
	
\textit{\textbf{Case 1.}} Let  $\ell=0$. Using Proposition \ref{proposition1} we have the following brackets:
	$$\begin{array}{lll}
	\bl h,v_i\br =\eta^{11}_iv_i+\eta_{i}^{12}w_{i},  & \bl h,w_i\br =\eta^{21}_{i}v_{i}+\eta^{22}_i w_i, &  0\leq i\leq n.\\ [1mm]
	\end{array}
	$$
	Identity (\ref{[h,[m,f]]}) for $v=v_i, \ 0\leq i\leq n-1$  yields
	$$\eta^{11}_{i+1}v_{i+1}+\eta_{i+1}^{12}w_{i+1}=(\eta^{11}_i-2\phi^{11})v_{i+1}+(\eta_i^{12}-2\phi^{12})w_{i+1}.$$
	From this we have $\eta^{11}_{i+1}=\eta^{11}_{i}-2\phi^{11}$ and $ \eta_{i+1}^{12}=\eta_i^{12}-2\phi^{12}$ for all $0\leq i\leq n-1.$ Denoting by $\eta^{11}:=\eta^{11}_0$ and $\eta^{12}\colon=\eta_0^{12}$ derives 
	$$\bl h,v_i\br =(\eta^{11}-2i\phi^{11})v_i+(\eta^{12}-2i\phi^{12})w_i, \  0\leq i\leq n.$$
\noindent	Similarly, $\bl h,w_i\br =(\eta^{21}-2i\phi^{21})v_i+(\eta^{22}-2i\phi^{22})w_i, \,\, 0\leq i\leq n$.
	
\textit{\textbf{Case 2.}}	For $\ell=1$ by Proposition \ref{proposition1} there are the following brackets:
	\[
	\begin{array}{lll}
	\bl h,v_0\br =\eta^{11}_0 v_0, &  0\leq i \leq m, \\ [1mm]
	\bl h,v_i\br =\eta^{11}_iv_i+\eta_{i-1}^{12}w_{i-1}, & 1\leq i\leq m+1=n-1,  \\ [1mm]
	\bl h,v_n\br =\eta^{11}_n v_n,  \\ [1mm]
	\bl h,w_i\br =\eta^{21}_{i+1}v_{i+1}+\eta^{22}_iw_i, & 0\leq i \leq m. & \\
	\end{array}
	\]
	
For all $0\leq i \leq m$ identity (\ref{[h,[m,f]]}) with $v=v_i$ implies $\eta^{11}_{i+1}=\eta^{11}_i-2\phi^{11}$ and $\eta_i^{12}=\eta_{i-1}^{12}-2\phi^{12}$. Therefore, $\eta^{11}_i=\eta^{11}_0-2i\phi^{11}$ and $\eta_i^{12}=-2(i+1)\phi^{12}$ for all $0\leq i \leq m$. Identity (\ref{[h,[m,f]]}) for $v=v_{m+1}$ implies 	$\eta^{11}_{n}=\eta^{11}_{n-1}-2\phi^{11}$. 	Denoting $\eta^{11}:=\eta^{11}_0$ yields $ \bl h,v_i\br =(\eta^{11}-2i\phi^{11})v_i-2i\phi^{12} w_{i-1}$ for $0\leq i \leq n$.
	
Consider identity (\ref{[h,[m,f]]}) for $v=w_i$ and $0\leq i \leq m-1$  to obtain $\eta^{21}_{i+2}=\eta^{21}_{i+1}-2i\phi^{21}$ and $  \eta^{22}_{i+1}=\eta^{22}_i-2\phi^{22}$.

Same identity for $v=w_m$ implies  $\eta^{21}_{m+1}=2\phi^{21}$.
Then
\[\eta^{21}_{i+1}=2(m+1-i)\phi^{21}, \quad \eta^{22}_i=\eta^{22}_0-2i\phi^{22}, \quad 0\leq i \leq m.\]

\textit{\textbf{Case 3.}} For $\ell\geq2$ by Proposition \ref{proposition1} there are the following brackets:
		$$\begin{array}{lll}
		\bl h,v_i \br =\eta_i^{11} v_i , \,\, 0\leq i\leq \ell-1, & \\ 	[1mm]
	
	\bl h,v_i \br =\eta_i^{11} v_i  +\eta_{-\ell+i}^{12}w_{-\ell+i}, \,\, \ell\leq i\leq m+\ell, & \\ [1mm]
	
	\bl h,v_i \br =\eta_i^{11} v_i, \,\, m+\ell+1\leq i\leq n, & \\ [1mm]
	
	\bl h,w_j\br =\eta_{\ell+j}^{21}v_{\ell+j} +\eta_{j}^{22}w_j, \,\, 0\leq j\leq m. &
	\end{array}$$

	Considering identity (\ref{[h,[m,f]]}) for $v=v_i$, $0\leq i\leq n-1$ implies 
	$\eta^{11}_i=\eta^{11}_0-2i\phi^{11}$ for $1\leq i \leq n$ and $\eta_j^{12}=0$ for all $0\leq j \leq m$. Substitute $\eta^{11}:=\eta_0^{11}$ to obtain $\bl h,v_i\br =(\eta^{11}-2i\phi^{11})v_i$ for all $0 \leq i \leq n$.

	Next, identity (\ref{[h,[m,f]]}) for $v=w_i, \ 0\leq i\leq m-1$ implies 
		\[\eta^{21}_{\ell}=\eta^{21}_{\ell+1}=\dots=\eta^{21}_{\ell+m}, \quad \eta^{22}_i=\eta^{22}_0-2\phi^{22} i, \quad 1\leq i \leq m.	\]
	 For $i=m$ we get $\eta^{21}_{\ell+m}=0$. Denoting by $\eta^{22}:=\eta^{22}_0$ yields $\bl h,w_j\br =(\eta^{22}-2\phi^{21}j)w_j$ for all $0\leq j \leq m$, and the proof is complete. 
 \end{proof}

 By Propositions \ref{proposition_f}, \ref{proposition_e} and \ref{proposition_h} for the case $\ell\geq 2$ we have the following table of brackets:
$$	\begin{array}{ll}
\,	 \bl f,v_i\br =\phi^{11} v_{i+1}   & 0\leq i \leq n-1\\
\,	 \bl f,w_j\br =\phi^{22} w_{j+1}           & 0\leq j \leq m-1\\		
\,	 \bl e,v_i\br =\displaystyle\frac{i(n-i+1)}{n}\epsilon^{11}v_{i-1} &         0\leq i \leq n-1    \\ 
\,	 \bl e,w_j\br =\displaystyle\frac{j(m-j+1)}{m}\epsilon^{22}w_{j-1} &       0\leq j \leq m-1             \\
\,	 \bl h,v_i\br =(\eta^{11}-2i\phi^{11})v_i & 0\leq i \leq n \\ 
\,	 \bl h,w_i\br =(\eta^{22}-2i\phi^{21})w_i & 0\leq i \leq m 
\end{array}.$$	 
The $\mathfrak{sl}_2$-modules $V$ and $W$ become $\mathfrak{sl}_2$-bimodules and the following is immediate. 
\begin{corollary}\label{l_geq_2}
	Let $\ell \geq 2$. Then $M$ is decomposable $\mathfrak{sl}_2$-bimodule.
\end{corollary}

\begin{proposition}\label{l=1}
	Let $\ell=1.$ Then the following holds:
	$$\begin{array}{llll}
	\bl h,v_i\br =(n-2i)\phi^{11}v_i-2i\phi^{12}w_{i-1}, \ 0 \leq i \leq n,  \\ [1mm]
	\bl f,v_i\br =\phi^{11}v_{i+1}+\phi^{12}w_i, \ 0\leq i \leq n-1, \\ [1mm]
	\bl e,v_i\br =-i(n-i+1)\phi^{11}v_{i-1} + i(i-1)\phi^{12}w_{i-2}, \ 1\leq i \leq n, \\  [1mm]
	\bl h,w_i\br =2(m-i+1)\phi^{21}v_{i+1}+(m-2i)\phi^{22}w_{i}, \ 0 \leq i \leq m,  \\ [1mm]
	\bl f,w_i\br =\phi^{21}v_{i+2}+\phi^{22}w_{i+1}, \ 0\leq i \leq m, \\ [1mm]
	\bl e,w_i\br =(m-i+1)((m-i+2)\phi^{21}v_{i} -i\phi^{22}w_{i-1}), \ 0\leq i \leq m. \\
	\end{array}$$

\end{proposition}

\begin{proof}
	Propositions \ref{proposition_f}, \ref{proposition_e} and \ref{proposition_h} for $\ell=1$ provides the following brackets:
$$	\begin{array}{ll}
\,		\bl f,v_i\br =\phi^{11} v_{i+1}+\phi^{12} w_i & 0\leq i \leq n-1\\
\,		\bl f,w_j\br =\phi^{21}v_{i+2}+\phi^{22} w_{j+1}               & 0\leq j \leq m-1\\		
\,	 \bl e,v_i\br = i\left( \displaystyle\frac{(n-i+1)}{n}\epsilon^{11}v_{i-1}+ \frac{i+1}{2}\epsilon^{12}w_{i-2} \right) &         0\leq i \leq n-1    \\ 
\,	 \bl e,w_j\br =(m+1-j)\left(\displaystyle\frac{m+2-j}{(m+2)(m+1)}\epsilon^{21}v_j+\frac{j}{m}\epsilon^{22}w_{j-1}\right) &       0\leq j \leq m-1             \\
\,	 \bl h,v_i\br =(\eta^{11}-2i\phi^{11})v_i-2i\phi^{12} w_{i-1} & 0\leq i \leq n \\ 
\,	\bl h,w_i\br =2(m+1-i)\phi^{21} v_{i+1}+(\eta^{22}-2i\phi^{22})w_i & 0\leq i \leq m
\end{array}.$$

Identity (\ref{[h,[m,e]]}) for $v=v_i, \ 2 \leq i \leq n$ implies 	$\epsilon^{11}=-n\phi^{11}$ and $\displaystyle \epsilon^{12}=\frac{2(i-1)}{i+1}\phi^{12}.$
	Thus we have  $\bl e,v_i\br =-i(n-i+1)\phi^{11} v_{i-1}+i(i-1)\phi^{12}w_{i-2}$ for $1\leq i \leq n.$
	
	Similarly, considering identity (\ref{[h,[m,e]]}) for $v=w_0$ and $v=w_1$ yields $\epsilon^{21}=n(m+1)\phi^{21}$ and $\epsilon^{22}=-m\phi^{22}$, correspondingly.

One gets $\eta^{11}=n\phi^{11}$ using identity (\ref{[f,[m,e]]}) for $v=v_n$. This implies $\bl h,v_i\br =(n-2i)\phi^{11}v_i-2i\phi^{12}w_{i-1}$ for all $0\leq i \leq n$.
	
	Consider identity (\ref{[f,[m,e]]}) for $v=w_m$ to obtain $\eta^{22}=m\phi^{22}$ which completes the proof of the proposition.	
\end{proof}

\begin{proposition}\label{l=0} Let $\ell=0$. Then the following holds:
	$$\begin{array}{lll}
	\bl h,v_i\br =(n-2i)(\phi^{11} v_i+\phi^{12} w_i), \,\, 0\leq i\leq n, & \\ [1mm]
	
	\bl f,v_i\br =\phi^{11} v_{i+1}+\phi^{12} w_{i+1}, \,\, 0\leq i\leq n-1, & \\ [1mm]
	
	\bl e,v_i\br =-i(n-i+1)(\phi^{11} v_{i-1}+\phi^{12} w_{i-1}), \,\, 1\leq i\leq n, \\
	[3mm]
	
	\bl h,w_i\br =(n-2i)(\phi^{21} v_i+\phi^{22} w_i), \,\, 0\leq i\leq n, &\\
	[1mm]
	
	\bl f,w_i\br =\phi^{21} v_{i+1}+\phi^{22} w_{i+1}, \,\, 0\leq i\leq n-1, & \\
	[1mm]
	\bl e,w_i\br =-i(n-i+1)(\phi^{21} v_{i-1}+\phi^{22} w_{i-1}), \,\, 1\leq i\leq n. \\
	\end{array}$$
\end{proposition}

\begin{proof}

	By Propositions \ref{proposition_f}, \ref{proposition_e} and \ref{proposition_h} for the case $l=0$ we have following:
	$$\begin{array}{llll}
	\bl h,v_i\br =(\eta^{11}-2i\phi^{11})v_i+(\eta^{12}-2i\phi^{12})w_i, \,\, 0\leq i\leq n,\\ [1mm]
	\bl f,v_i\br =\phi^{11} v_{i+1}+\phi^{12} w_{i+1}, \,\, 0\leq i\leq n-1, \\ [1mm]
	\bl e,v_i\br =\displaystyle\frac{i(n-i+1)}{n}(\epsilon^{11} v_{i-1}+\epsilon^{12} w_{i-1}), \,\, 1\leq i\leq n,\\ [1mm]
	\bl h,w_i\br =(\eta^{21}-2i\phi^{21})v_i+(\eta^{22}-2i\phi^{22})w_i, \,\, 0\leq i\leq n,\\ [1mm]
	\bl f,w_i\br =\phi^{21} v_{i+1}+\phi^{22} w_{i+1}, \,\, 0\leq i\leq n-1, \\ [1mm]
	\bl e,w_i\br =\displaystyle\frac{i(n-i+1)}{n}(\epsilon^{21}v_{i-1}+\epsilon^{22} w_{i-1}), \,\, 1\leq i\leq n.\\
	\end{array}$$
	
	Consider identity (\ref{[f,[m,e]]}) for $m=v_n$ and $m=w_n$ to obtain $\eta^{11}=\phi^{11} n,\eta=\phi^{12} n$ and $\eta^{21}=\phi^{21} n, \eta^{22}=\phi^{22} n$, correspondingly. Analogously, identity (\ref{[h,[m,e]]}) for $m=v_n$ and  $m=w_n$ implies $\epsilon^{11}=-\phi^{11} n, \epsilon^{12}=-\phi^{12} n$ and $\epsilon^{21}=-\phi^{21} n, \epsilon^{22}=-\phi^{22} n$, correspondingly. This completes the proof. 
\end{proof}

\section{Results from identity (\ref{blue})}

\begin{proposition}\label{l=0_blue}
    Let $n=m.$ Then $M$ is split  as a Leibniz bimodule. 
\end{proposition}

\begin{proof}
    Using Proposition \ref{l=0} let us verify the identity (\ref{blue}): 
    
    $0=\bl h,\rl v_0,h\rr +\bl h,v_0\br \br =n^2 ((1+\phi^{11} )\phi^{11} +\phi^{12} \phi^{21} )v_0+n^2((1+\phi^{11} )\phi^{12} +\phi^{12} \phi^{22} )w_0$,
    
        $0=\bl h,\rl w_0,h\rr +\bl h,w_0\br \br =n^2 (\phi^{21} \phi^{11} +(1+\phi^{22} )\phi^{21} )v_0+n^2 (\phi^{21} \phi^{12} +(1+\phi^{22} )\phi^{22} )w_0$.
        
   Thus, we have
    \begin{equation}\label{system}
    \left\{%
    \begin{array}{ll}
    (1+\phi^{11} )\phi^{11} +\phi^{12} \phi^{21} =0, \\
    \phi^{12} (1+\phi^{11} +\phi^{22} )=0, \\
    \phi^{21} (1+\phi^{11} +\phi^{22} )=0, \\
    \phi^{21} \phi^{12} +(1+\phi^{22} )\phi^{22} =0.
    \end{array}%
    \right.
        \end{equation}
    
    Note that, for $M$ to be indecomposable, $\phi^{12} $ and $\phi^{21} $ cannot be simultaneously zero. Therefore $1+\phi^{11} +\phi^{22} =0$ and substituting $\phi^{22}  = -(1+\phi^{11} )$ we obtain $\phi^{12}  \phi^{21}  = \phi^{11}  \phi^{22} $. Let us consider several cases.
    
\textit{\textbf{Case 1.}} Let $\phi^{12} \neq 0$. Set $x_i=(1+\phi^{11} )v_i+\phi^{12}  w_i$ and $y_i=\phi^{11}  v_i+\phi^{12}  w_i$ for $0\leq i \leq n$. Note that $X=\Span \{x_i\}_{0\leq i \leq n}, Y=\Span \{y_i\}_{0\leq i \leq n}$ are simple $\mathfrak{sl}_2$-modules since identities from Theorem \ref{irreducible} hold. Moreover, they are subbimodules of $M$ due to identities from (\ref{system}):
\begin{align*}
\bl h,x_i\br=&\bl h, (1+\phi^{11} )v_i+\phi^{12}  w_i\br& \\
=&((1+\phi^{11} )\phi^{11}  +\phi^{12} \phi^{21} )(n-2i) v_i + ((1+\phi^{11} +\phi^{22} )\phi^{12} )(n-2i)w_i=0,&\\
\bl h,y_i\br=&\bl h, \phi^{11}  v_i+\phi^{12}  w_i\br&\\
=&(\phi^{11}\phi^{11} +\phi^{12} \phi^{21} )(n-2i) v_i + (\phi^{11} +\phi^{22} )\phi^{12} (n-2i)w_i=-(n-2i)y_i.&
\end{align*}
Similarly, $\bl f,x_i\br=\bl e,x_i\br=0$ and $ \bl f,y_{i-1}\br=-y_i,\ \bl e,y_{i+1}\br=(i+1)(n-i)y_i$. Hence, $M$ is a direct sum of $X$ and $Y$ as subbimodules. 

\textit{\textbf{Case 2.}} Assume  $\phi^{12} =0$. Then $\phi^{11}  \phi^{22} =0$. Consider the following sub-cases.

\textit{Case 2.1.} Let $\phi^{11} =0$. Then  $\phi^{22} =-1$ and we have $\bl h,v_i\br =\bl f,v_i\br =\bl e,v_i\br =0$ for all $0\leq i\leq n$, and 
\begin{center}
	$\begin{array}{lll}
    \bl h,w_i\br =(n-2i)(\phi^{21}  v_i- w_i), \,\, 0\leq i\leq n, &\\
    [1mm] 

    \bl f,w_i\br =\phi^{21}  v_{i+1}-w_{i+1}, \,\, 0\leq i\leq n-1, & \\
    [1mm]
    \bl e,w_i\br =-i(n-i+1)(\phi^{21}  v_{i-1}-w_{i-1}), \,\, 1\leq i\leq n. \\
    \end{array}$

\end{center}
    
 Substitute $u_i=\phi^{21}  v_i- w_i$ for all $1\leq i \leq n$ and set $U=\Span\{u_i\}_{0\leq i \leq n}$. Note that $U$ is a Leibniz subbimodule of $M$ since $\rl u_i,h\rr=-\bl h,u_i\br=(n-2i)u_i, \ \rl u_{i-1},f\rr=-\bl f,u_{i-1}\br=u_i$ and $\rl u_{i+1},e\rr=\bl e,u_{i+1}\br =u_i$. Moreover, $M$ decomposes as the direct sum of its subbimodules $V$ and $U$. 

\textit{Case 2.2.} Let $\phi^{22} =0$. Then $\phi^{11} =-1$ and  for all $ 0\leq i\leq n$ we have
    $$\begin{array}{lll}
    \bl h,v_i\br=-(n-2i)v_i, & \bl h,w_i\br=(n-2i)\phi^{21}  v_i,\\ [1mm]

    \bl f,v_{i-1}\br=-v_i,  &  \bl f,w_{i-1}\br=\phi^{21}  v_i\\ [1mm]

   \bl e,v_{i+1}\br=(i+1)(n-i)v_i, & \bl e,w_{i+1}\br=\phi^{21} (i+1)(n-i)v_i, \\
    \end{array}$$
Similarly, $u_i=\phi^{21}  v_i+ w_i$ for all $1\leq i \leq n$ and set $U=\Span\{u_i\}_{0\leq i \leq n}$. Note that $U$ is a Leibniz subbimodule of $M$ since $\rl u_i,h\rr=(n-2i)u_i, \ \rl u_{i-1},f\rr=u_i, \ \rl u_{i+1},e\rr=-(i+1)(n-i)u_i  $ and $\bl h,u_i\br =\bl f,u_i\br =\bl e,u_i\br =0$. Moreover, $M$ decomposes as the direct sum of its subbimodules $V$ and $U$. 
\end{proof}

The following theorem characterizes the only indecomposable Leibniz $\mathfrak{sl}_2$-bimodules if we assume that as a Lie $\mathfrak{sl}_2$-module it is a direct sum of two simple $\mathfrak{sl}_2$-submodules. 
\begin{theorem}\label{main}
	An $\mathfrak{sl}_2$-module $M=V\oplus W$, where $V$ and $W$ are simple $\mathfrak{sl}_2$-modules is indecomposable as a Leibniz $\mathfrak{sl}_2$-bimodule if and only if $\dim V-\dim W=2$. Moreover, up to $\mathfrak{sl}_2$-bimodule isomorphism there are only two indecomposable $\mathfrak{sl}_2$-bimodules, which in basis $\{v_0,\dots,v_n,w_0,\dots, w_{n-2}\}$ have the following brackets:
	\begin{align*}
	&{\color{red}[}v_i , h{\color{red}]} \ = (n-2i)v_i  &           	\bl h,v_i \br &=-(n-2i)v_i -2iw_{i-1}  \\ 
	&{\color{red}[}v_{i} ,f{\color{red}]} \ =  v_{i+1}  &            \bl f,v_i\br&=-v_{i+1}+w_i\\ 
 M_1: \hspace{1cm}	&{\color{red}[}v_{i} ,e{\color{red}]} \, \ = -i(n-i+1)v_{i-1} &     \bl e,v_{i}\br&=i(n-i+1)v_{i-1}+i(i-1)w_{i-2}\\
	& {\color{red}[}w_j, h{\color{red}]} = (n-2-2j)w_j     &       \bl h,w_{j}\br&=0\\
	&  {\color{red}[}w_{j},f{\color{red}]} = w_{j+1} &  \bl f,w_{j}\br&=0\\ 
	&  {\color{red}[}w_{j},e{\color{red}]} \, = -j(n-1-j)w_{j-1} &  \bl e,w_{j}\br&=0
 	\end{align*}

	\begin{align*}
	&{\color{red}[}v_i , h{\color{red}]} \ = (n-2i)v_i  &           	\bl h,v_i \br &=0  \\ 
	&	{\color{red}[}v_{i} ,f{\color{red}]} \ =  v_{i+1}  &            \bl f,v_i\br&=0\\ 
 M_2: \hspace{1cm}	&{\color{red}[}v_{i} ,e{\color{red}]} \, \ = -i(n-i+1)v_{i-1} &     \bl e,v_{i}\br&=0\\
	& {\color{red}[}w_j, h{\color{red}]} = (n-2-2j)w_j     &       \bl h,w_{j}\br&=2(m-j+1)v_{j+1}-(m-2j)w_{j}\\
	&  {\color{red}[}w_{j},f{\color{red}]} = w_{j+1} &  \bl f,w_{j}\br&=v_{j+2}-w_{j+1}\\ 
	&  {\color{red}[}w_{j},e{\color{red}]} \, = -j(n-1-j)w_{j-1} &  \bl e,w_{j}\br&=(m-j+1)((m-j+2)v_{j} +iw_{j-1})
	\end{align*}

\end{theorem}

\begin{proof} By Corollary \ref{l_geq_2} and Propositions \ref{l=0_blue} it follows that unless $\ell=1$ bimodule $M$ is split. Therefore, an indecomposable bimodule might appear only if the difference in dimensions of $V$ and $W$ is 2. Using Proposition \ref{l=1} 	we have the following:
	$$\begin{array}{llll}
	\bl h,v_i\br =(n-2i)\phi^{11}v_i-2i\phi^{12}w_{i-1}, \ 0 \leq i \leq n,  \\ [1mm]
	\bl f,v_i\br =\phi^{11}v_{i+1}+\phi^{12}w_i, \ 0\leq i \leq n-1, \\ [1mm]
	\bl e,v_i\br =-i(n-i+1)\phi^{11}v_{i-1} + i(i-1)\phi^{12}w_{i-2}, \ 1\leq i \leq n, \\  [1mm]
	\bl h,w_i\br =2(m-i+1)\phi^{21}v_{i+1}+(m-2i)\phi^{22}w_{i}, \ 0 \leq i \leq m,  \\ [1mm]
	\bl f,w_i\br =\phi^{21}v_{i+2}+\phi^{22}w_{i+1}, \ 0\leq i \leq m, \\ [1mm]
	\bl e,w_i\br =(m-i+1)((m-i+2)\phi^{21}v_{i} -i\phi^{22}w_{i-1}), \ 0\leq i \leq m. \\
	\end{array}$$

 From identity (\ref{blue}) for $(f,v_0,f)$ and $(f,w_0,f)$ we obtain system (\ref{system}). Note that, if $1+\phi^{11}+\phi^{22}\neq 0$ then $\phi^{12}=\phi^{21}=0$, that is $M$ is decomposable as a Leibniz $\mathfrak{sl}_2$-module. 
 
 Therefore, we assume that  
 \begin{equation}\label{1+phi+phi}
 1+\phi^{11}+\phi^{22}= 0
 \end{equation}

From identity (\ref{blue}) we have 
 $$0=\bl f, \rl w_0,e\rr+\bl e,w_0\br \br=\bl f, (m+1)(m+2)\phi^{21}v_0\br=(m+1)(m+2)\phi^{21} \bl f , v_0 \br=(m+1)(m+2)\phi^{21}(\phi^{11}v_1+\phi^{12}w_0),$$	
which implies $\phi^{21}\phi^{11}=\phi^{21}\phi^{12}=0$.

\textbf{Case 1.} Let $\phi^{21}=0$. Let us consider the following subcases.

\

\textit{Case 1.1.} Let $\phi^{22}\neq 0$. Then from the last equation of system (\ref{system}) it follows that $\phi^{22}=-1$. Furthermore, system (\ref{system}) implies $\phi^{11}=0$.  From identity (\ref{blue}) for triple $(f,h, v_0)$ it follows that $\phi^{12}=0$.  Then $M$ is decomposable as a direct sum of $\mathfrak{sl}_2$-subbimodules $V$ and $W$. 
	
	\
	
\textit{Case 1.2.} Let 	$\phi^{22}= 0$. Then $\bl \mathfrak{sl}_2, W \br =0$. Using Proposition \ref{l=1} and identity (\ref{blue}) we have:
$$0=\bl h, \rl v_0,f\rr+\bl f,v_0\br \br=\bl h, (\phi^{11}+1)v_1+w_1\br=(\phi^{11}+1)((n-2)\phi^{11}v_1-2\phi^{12}w_0),$$	
which implies $\phi^{11}=-1$, since otherwise $M$ is not indecomposable. Note that, if $\phi^{12}=0$, then $V$ becomes a subbimodule and $M$ is decomposable. Therefore, rescaling the basis of $W$ by $\phi^{12}$ we obtain the product of $M_1$ from the statement of the theorem. 

$$\begin{array}{llll}
\bl h,v_i\br=-(n-2i)v_i-2iw_{i-1}, \ 0 \leq i \leq n,  \\ [1mm]
\bl f,v_i\br=-v_{i+1}+w_i, \ 0\leq i \leq n-1, \\ [1mm]
\bl e,v_i \br=i(n-i+1)v_{i-1} + i(i-1)w_{i-2}, \ 1\leq i \leq n \\
\end{array}$$

Let us prove that the only $\mathfrak{sl}_2$-subbimodule $N$ of $M_1$ in this case is $W$. Since $W$ is a simple $\mathfrak{sl}_2$-module, if $N \subseteq W$ then $N$ is trivial or $N=W$. Let us assume that $u=\displaystyle \sum_{i=0}^n \alpha_i v_i +\sum_{j=0}^{n-2}\beta_j w_j$ is an element of $N$ with not all $\alpha_i=0$.  Let $k$ be the largest index such that $\alpha_k\neq 0$. Then $\bl e,\rl \dots\rl \rl u,\underbrace{e\rr ,e\rr,\dots,e\rr}_{k \text{ times}}\br$ is a nonzero multiple of $v_0$. Hence, $v_0\in N$ and acting with $f$ from the right we obtain $V\subseteq N$. From $\bl f,v_0\br =-v_1+w_0$ it follows that $w_0\in N$. Similarly, $W\subseteq N$. Thus $N=M_1$ and the $\mathfrak{sl}_2$-bimodule $M_1$ is indecomposable. 

\

\textbf{Case 2.} Let $\phi^{21}\neq 0$. Then $\phi^{12}=0$. From indentity (\ref{blue}) we have:
$$0=\bl h, \rl w_0,e\rr+\bl e,w_0\br \br=\bl h, (m+1)(m+2)\phi^{21}v_0\br=(m+1)(m+2)\phi^{21} \bl h , v_0 \br=(m+1)(m+2)n\phi^{21}\phi^{11}v_1$$ 
that implies $\phi^{11}=0$. From (\ref{1+phi+phi}) it follows that $\phi^{22}=-1$. 

Rescaling the basis vectors of $V$ to $\phi^{21}$ we obtain the following:

$$\begin{array}{llll}
\bl h,v_i\br =\bl f,v_i\br =\bl e,v_i\br=0,\\
\bl h,w_i\br =2(m-i+1)v_{i+1}-(m-2i)w_{i}, \ 0 \leq i \leq m,  \\ [1mm]
\bl f,w_i\br =v_{i+2}-w_{i+1}, \ 0\leq i \leq m, \\ [1mm]
\bl e,w_i\br =(m-i+1)((m-i+2)v_{i} +iw_{i-1}), \ 0\leq i \leq m.\\
\end{array}$$
As in the first case, one can show that this $\mathfrak{sl}_2$-bimodule is indecomposable and we denote it by $M_2$. 
\end{proof}

\begin{remark} Note that in Theorem \ref{main} the $\mathfrak{sl}_2$-bimodule $M_1$ admits antisymmetric  $\mathfrak{sl}_2$-subbimodule $W$ and $M_1/W$ is a symmetric $\mathfrak{sl}_2$-bimodule. For $M_2$ the antisymmetric $\mathfrak{sl}_2$-subbimodule is $V$ and the quotient $M_2/V$ is a symmetric $\mathfrak{sl}_2$-bimodule. This is in accordance with the result of \cite[Section 4]{LP_Leib_rep} that claims the group $Ext^1_{UL(\mathfrak{sl}_2)}(M,N)=0$ for simple $\mathfrak{sl}_2$-bimodules $M$ and $N$ unless $M$ and $N$ are symmetric and antisymmetric $\mathfrak{sl}_2$-bimodules coresspondingly, of dimensions $n$ and $n-2$, or of dimensions $n$ and $n+2$, correspondingly.  Two  $\mathfrak{sl}_2$-bimodules of Theorem \ref{main} correspond to these nontrivial extensions of \cite{LP_Leib_rep}.
\end{remark}


\end{document}